\documentclass[12pt]{article}
\usepackage{amsmath, amssymb}
\usepackage{hyperref}
\usepackage{tikz}
\usepackage[latin1]{inputenc}
\usepackage{amsmath}
\usepackage{amsfonts}
\usepackage{amssymb}
\usepackage{amsthm}
\usepackage{mathrsfs} 
\usepackage{mathtools}
\usepackage{graphicx}
\usepackage[square,sort,comma,numbers]{natbib}
\usepackage[left=1.25in,right=1in,top=1in,bottom=1in]{geometry}
\usepackage{amsfonts}
\usepackage{xcolor}

\usepackage{setspace}

\usepackage[english]{babel}
\usepackage{hyperref}
\usepackage{natbib}
\hypersetup{
	colorlinks=true,
	linkcolor=blue,
	citecolor=blue,
	urlcolor=blue,
}

\theoremstyle{definition}
\newtheorem{definition}{Definition}[section]
\newtheorem{theorem}{Theorem}[section]
\newtheorem{corollary}{Corollary}[section]
\newtheorem{lemma}{Lemma}[section]
\newtheorem{proposition}{Proposition}
\theoremstyle{remark}
\newtheorem{remark}{Remark}[section]

\theoremstyle{definition}
\newtheorem{example}{Example}[section]

\begin{document}
	
	\title{On the Eccentricity Laplacian and Eccentricity Signless Laplacian Matrices of a Graph}
	\author{Keshav Saini \footnote{identity.keshav@gmail.com, Malaviya National Institute of Technology, Jaipur, Rajasthan, India 302017} \and Anubha Jindal \footnote{anubha.maths@mnit.ac.in, Malaviya National Institute of Technology Jaipur, Rajasthan, India 302017} \and K. Palpandi \footnote{kpandiiitm@gmail.com, National Institute of Technology Calicut, Kerala, India, 673601}}
	\date{}
	\maketitle
	\begin{abstract}
In this paper, we introduce the Laplacian and the signless Laplacian for the eccentricity matrix of a connected graph, referred to as the eccentricity Laplacian and the eccentricity signless Laplacian, respectively. We establish the equivalence among the eccentricity Laplacian, eccentricity signless Laplacian, and eccentricity spectrum for different classes of graphs. We provide spectral characterization of $\mathcal{E}$-bipartite graphs  by the symmetry of $\mathcal{E}$-spectrum and the similarity of these Laplacian matrices.	
	\end{abstract}
	
	\section{Introduction}
	We begin by reviewing a few key definitions. Only simple, undirected, and finite graphs are examined in this study. \(G = G(V, E)\), where \( V \) is the graph's vertex set, and \( E \) is its edge set, are the standard notations for a graph. The \textit{order} of a graph is the number of its vertices.
 The adjacency matrix \( \mathcal{A}(G) \) of \( G \) is a \( 0\text{-}1 \) \( n \times n \) matrix that is defined by $ a_{ij} = 1 \quad \text{if and only if } v_iv_j \in E,$ and indexed by the vertices of \( G \). A vertex is said to be \textit{universal vertex}, if it is adjacent to all the remaining vertices. The number of vertices which are adjacent to $v_i$ is the \textit{degree} of $v_i$, and is denoted by $deg(v_i)$. Let $U(G)$ denotes the collection of universal vertices, i.e. $U(G)=\{v_i\in V(G): deg(v_i)=n-1\}$. Let $G$ and $H$ be two graphs, then the \textit{join} of $G$ and $H$ is denoted by $G\vee H$, and defined as each vertex of $G$ is adjacent to every vertex of $H$. 
  
 The Laplacian of \( G \) is the matrix $\mathcal{L}(G) = \text{Diag}(\text{Deg($G$)}) - \mathcal{A}(G),$ where \( \text{Diag}(\text{Deg($G$)}) \) is the diagonal matrix whose diagonal entries are the degrees of the vertices. The \textit{signless Laplacian} of \( G \) is the matrix $ \mathcal{Q}(G) = \text{Diag}(\text{Deg($G$)}) + \mathcal{A}(G)$. Let $\{\mu_1, \mu_2, \dots, \mu_n\}$ and $\{q_1, q_2, \dots, q_n \}$ denote the eigenvalues of $\mathcal{L}(G)$, and $\mathcal{Q}(G)$, respectively. The second smallest eigenvalue of $\mathcal{L}(G)$ is known as  \textit{algebraic connectivity} \cite{Fiedler1973} and the associated eigenvector is \textit{Fiedler vector} \cite{Fiedler1975}. For further details on the signless Laplacian matrix and its spectrum, we refer the reader to \cite{CvetkovicI, CvetkovicII, CvetkovicIII}.

  The distance (the length of a shortest path) between two vertices, \( v_i \) and \( v_j \), in a connected graph, \( G \), is denoted by the \( d(v_i,v_j) \). The \textit{distance matrix} \( \mathcal{D}(G) \) whose rows and columns are indexed by vertices of a connected graph \( G \), is defined as $ \mathcal{D}(G)_{i,j} = d(v_i, v_j)$. The \textit{transmission} $\mathrm{Tr}(v_i)$ of a vertex $v_i$ is the sum of the distances from $v_i$ to all the remaining vertices in $G$. For the distance matrix, Aouchiche and Hansen \cite{Aouchiche} introduced the distance Laplacian and distance signless Laplacian matrices of a graph $G$, denoted by $\mathcal{D}^{\mathcal L}(G)$ and $\mathcal{D}^{\mathcal Q}(G)$, respectively, and defined by $
\mathcal{D}^{\mathcal L}(G)=\operatorname{Diag}(\mathrm{Tr}(G))-\mathcal{D}(G),$ and $
\mathcal{D}^{\mathcal Q}(G)=\operatorname{Diag}(\mathrm{Tr}(G))+\mathcal{D}(G),
$
where $\operatorname{Diag}(\mathrm{Tr}(G))$ is the diagonal matrix whose diagonal entries are the transmissions of the vertices of $G$. For more recent work on these matrices, we refer \cite{Aouchiche2014, Aouchiche2016, Aouchiche2017, Nath2014}.
  
  The \textit{eccentricity} of a vertex \( v_i \) is denoted by  \( e(v_i) \), is defined as follows: $ e(v_i) = \max \{ d(v_i, v_j): v_j \in V(G)\} $.  The \textit{diameter} and \textit{radius} of a graph $G$ are defined, respectively, as
$\mathrm{diam}(G) = \max \{ e(v_i): v_i \in V(G) \}, \quad
\mathrm{rad}(G) = \min \{ e(v_i): v_i \in V(G) \}.$ Let $C(G)=\{v_i\in V(G): e(v_i)=rad(G)\}$ and $P(G)=\{v_i\in V(G): e(v_i)=diam(G)\}$ be the collection of central and periphery vertices of $G$, respectively. In 2013, Randi\'c introduced the notion of the eccentricity matrix \cite{Randic2013} as the $D_{MAX}$-matrix and renamed it as the eccentricity matrix by Wang et al. \cite{Wang2018} in 2018.
	The eccentricity matrix $\mathcal{E}(G)$ of a connected graph $G$ is obtained from the distance matrix of $G$ by keeping maximal nonzero entries in each row and each column and leaving zeros in the remaining ones. Equivalently,\[\mathcal{E}_{ij}(G)=\begin{cases}
		d(v_i,v_j) ~~ \text{if} ~ d(v_i,v_j)=min\{e(v_i),e(v_j)\}\\
		~~~~0 ~~~~ \text{if} ~ d(v_i,v_j)<min\{e(v_i),e(v_j)\}.
	\end{cases}\]
     Although the eccentricity matrix is obtained from the distance matrix, it possesses several properties that differ from those of the distance matrix. In particular, the distance matrix of every connected graph is irreducible, whereas the eccentricity matrix of an even cycle is reducible and that of an odd cycle is irreducible. In \cite{Wang2018} and \cite{Mahato2019}, the authors proved that the eccentricity matrices of trees are irreducible.
     
For a matrix $A$, $\rho(A)$ and $\sigma(A)$ denote its \textit{spectral radius} and \textit{spectrum} respectively. The matrix $|A|$ is obtained by taking the absolute values of all entries of the matrix $A$. $I$, $J$, and $O$ denote the identity matrix, all-ones matrix, and zero matrix of appropriate order, respectively. The vectors $\mathbf{1}$ and $\mathbf{0}$ represent the all-ones vector and zero vector, respectively, while $|\mathbf{x}|$ denotes the vector obtained by taking the absolute value of each entry of $\mathbf{x}$.

The article is organized as follows: In Section~2, we introduce Laplacian-type matrices associated with the eccentricity matrix, investigate several properties of their eigenvalues, and derive the characteristic polynomials of these matrices for certain well-known graphs. In Section~3, we establish relationships between the spectra of various Laplacian-type matrices associated with the eccentricity matrix for different classes of graphs, and we study the eccentricity Laplacian matrix of graphs with diameter $2$ and discuss several consequences of these results. In Section~4, we define $\mathcal{E}$-bipartite graphs and find several characterizations for these graphs.

\section{Laplacians for the eccentricity matrix}

In this section, we define the Laplacian and the signless Laplacian for the eccentricity matrix. But first, we need the following definitions.
     
 The $\mathcal{E}$-transmission $\mathrm{Tr}_\mathcal{E}(v_i)$ of a vertex $v_i$ is defined to be the $i$-th row sum of the eccentricity matrix, that is, $$\mathrm{Tr}_\mathcal{E}(v_i)=\sum_{v_j\in V(G)}\mathcal{E}_{ij}(G).$$
    Let $\mathrm{Tr}_\mathcal{E}^{\max}$ and $\mathrm{Tr}_\mathcal{E}^{\min}$ denote the maximum and minimum $\mathcal{E}$-transmission values among the vertices of the graph.
	A connected graph $G$ is said to be $\mathcal{E}$-transmission regular of $\mathcal{E}$-degree $k$ if $\mathrm{Tr}_\mathcal{E}(v_i)=k$ for every $v_i\in V(G)$. Observe that $\mathcal{E}$-transmission regularity does not necessarily imply
degree regularity, as there exist graphs satisfying the former property but not the latter [see Figure 2].  The average $\mathcal{E}$-transmission is defined as $\overline{\mathrm{Tr}}_\mathcal{E}(G)=\frac{1}{n}\sum_{i=1}^{n}\mathrm{Tr}_\mathcal{E}(i)$. Note that $\overline{\mathrm{Tr}}_\mathcal{E}(G)=\frac{2}{n}W_\mathcal{E}(G)$, where the $\mathcal{E}$-Wiener index \cite{Mahato2023} of a connected graph $G$ is  $W_\mathcal{E}(G)=\frac{1}{2}\sum_{v_i,v_j\in V(G)}\mathcal{E}_{ij}(G)$, where $\mathcal{E}_{ij}(G)$ is the $ij$-th entry of $\mathcal{E}(G)$.

     Now we define the eccentric Laplacian of a connected graph $G$, denoted by $\mathcal{E}^\mathcal{L}(G)$, 
     
     $$\mathcal{E}^\mathcal{L}(G)={\rm Diag}(\mathrm{Tr}_\mathcal{E}(G))-\mathcal{E}(G),$$ where ${\rm Diag}(\mathrm{Tr}_\mathcal{E}(G))$ is the diagonal matrix of the $\mathcal{E}$-transmissions of the vertices in $G$. Similarly, the eccentric signless Laplacian of a connected graph $G$, denoted by $\mathcal{E}^\mathcal{Q}(G)$, is defined as
     $$\mathcal{E}^\mathcal{Q}(G)={\rm Diag}(\mathrm{Tr}_\mathcal{E}(G))+\mathcal{E}(G).$$ Observe that both $\mathcal{E^{L}}(G)$ and  $\mathcal{E^{Q}}(G)$ are positive semi-definite matrices. 
	
 Let $\{\xi_1, \xi_2, \dots, \xi_n\}$, $\{\xi_1^\mathcal{L}, \xi_2^\mathcal{L}, \dots, \xi_n^\mathcal{L}\}$, $\{\xi_1^\mathcal{Q},\xi_2^\mathcal{Q},\dots,\xi_n^\mathcal{Q}\}$ denote the spectrum of $\mathcal{E}(G)$, $\mathcal{E}^\mathcal{L}(G)$, and $\mathcal{E}^\mathcal{Q}(G)$, respectively. 	\begin{proposition}\label{first prop}
    For a connected graph $G$ of order $n$, the following assertions hold.
    \begin{enumerate}
        \item[(i)] $\sum_{i=1}^{n}\xi_i=0$ and $\sum_{i=1}^{n}\xi_i^\mathcal{L}=\sum_{i=1}^{n}\mathrm{Tr}_\mathcal{E}(i)=2W_{\mathcal{E}}(G)$. 

    \item[(ii)] $\sum_{i=1}^{n}(\xi_i)^2= 2e$, where $e=\sum_{1\leq i<j\leq n}(\mathcal{E}_{ij}(G))^2$.
     
    \item[(iii)] $\sum_{i=1}^{n}(\xi_i^\mathcal{L})^2= 2e+\sum_{i=1}^{n}(\mathrm{Tr}_\mathcal{E}(i))^2$.
    \end{enumerate}
    \end{proposition}
	\begin{proof}
\begin{enumerate}
    \item[(i)] This follows directly from the definition.
    \item[(ii)] $
\sum_{i=1}^{n}\xi_i^{2}
=\operatorname{tr}\big(\mathcal{E}(G)^{2}\big)
=\sum_{i,j=1}^{n}\mathcal{E}_{ij}(G)^{2}
=2e.$
    \item[(iii)]  \[
\begin{aligned}
\sum_{i=1}^{n}(\xi_i^{\mathcal L})^2
&=\operatorname{tr}\big((\mathcal{E}^{\mathcal L}(G))^2\big) 
=\operatorname{tr}\Big(\big(\operatorname{Diag}(\mathrm{Tr}_{\mathcal E}(G))-\mathcal{E}(G)\big)^2\Big) \\
&=\operatorname{tr}\Big(\operatorname{Diag}(\mathrm{Tr}_{\mathcal E}(G))^2
-\operatorname{Diag}(\mathrm{Tr}_{\mathcal E}(G))\mathcal{E}(G) \\
&\qquad -\mathcal{E}(G)\operatorname{Diag}(\mathrm{Tr}_{\mathcal E}(G))
+\mathcal{E}(G)^2\Big) \\
&=2e+\sum_{i=1}^{n}(\mathrm{Tr}_{\mathcal E}(i))^2.
\qedhere
\end{aligned} 
\]
 \end{enumerate}
\end{proof} 
\begin{remark}
For the Laplacian spectrum we have $\sum_{i=1}^{n}\mu_i^2=2m+\sum_{i=1}^{n}deg(v_i)^2$, where $m$ denotes the number of edges in the graph. Proposition \ref{first prop} (iii) refers to the similar result for the eccentric Laplacian.
 \end{remark} 
Since $\sum_{i=1}^{n}\xi_i=0$ and $\sum_{i=1}^{n}\xi_i^{2}=2e$, where $e$ is defined as earlier, we seek an analogous relation for $\mathcal{E}^{\mathcal L}(G)$. For this purpose, we define the auxiliary eigenvalues of $\mathcal{E}^{\mathcal L}(G)$.

The auxiliary eigenvalues of $\mathcal{E}^{\mathcal L}(G)$ are defined as
\[
\eta_i=\xi_i^{\mathcal L}-\frac{1}{n}\sum_{j=1}^{n}\mathrm{Tr}_{\mathcal E}(v_j)
      =\xi_i^{\mathcal L}-\overline{\mathrm{Tr}}_{\mathcal E}(G),
\qquad i=1,2,\ldots,n.
\]
\begin{proposition}\label{second prop}
Let $\eta_1,\eta_2,\dots,\eta_n$ be the auxiliary eigenvalues of $\mathcal{E^L}(G)$. Then
\begin{enumerate}
    \item[(i)] $\sum_{i=1}^{n}\eta_i=0.$
    \item[(ii)] $\sum_{i=1}^{n}\eta_i^2 = 2E$, where 
$E = e + \frac{1}{2}\sum_{i=1}^{n}\big(\mathrm{Tr}_{\mathcal E}(v_i)\big)^2
    - \frac{1}{2n}\left(\sum_{i=1}^{n}\mathrm{Tr}_{\mathcal E}(v_i)\right)^2$.

\end{enumerate}
\end{proposition}   
\begin{proof}
\begin{enumerate}
    \item[(i)] 
$\sum_{i=1}^{n}\eta_i
=\sum_{i=1}^{n}\left(\xi_i^{\mathcal L}-\overline{\mathrm{Tr}}_{\mathcal E}(G)\right)
=\sum_{i=1}^{n}\xi_i^{\mathcal L}-n\,\overline{\mathrm{Tr}}_{\mathcal E}(G)
=0.$ 
\item[(ii)] By definition, $\eta_i=\xi_i^{\mathcal L}-\overline{\mathrm{Tr}}_{\mathcal E}(G)$ for $i=1,2,\ldots,n$. Hence,
\[
\sum_{i=1}^{n}\eta_i^2
=\sum_{i=1}^{n}\big(\xi_i^{\mathcal L}-\overline{\mathrm{Tr}}_{\mathcal E}(G)\big)^2.
\]
Expanding the square, we obtain
\[
\sum_{i=1}^{n}\eta_i^2
=\sum_{i=1}^{n}\left[(\xi_i^{\mathcal L})^2+(\overline{\mathrm{Tr}}_{\mathcal E}(G))^2
-2\xi_i^{\mathcal L}\overline{\mathrm{Tr}}_{\mathcal E}(G)\right].
\]
Thus,
\[
\sum_{i=1}^{n}\eta_i^2
=\sum_{i=1}^{n}(\xi_i^{\mathcal L})^2
+n\big(\overline{\mathrm{Tr}}_{\mathcal E}(G)\big)^2
-2\overline{\mathrm{Tr}}_{\mathcal E}(G)\sum_{i=1}^{n}\xi_i^{\mathcal L}.
\]
Since $\sum_{i=1}^{n}\xi_i^{\mathcal L}=n\,\overline{\mathrm{Tr}}_{\mathcal E}(G)$, it follows that
\[
\sum_{i=1}^{n}\eta_i^2
=\sum_{i=1}^{n}(\xi_i^{\mathcal L})^2
-n\big(\overline{\mathrm{Tr}}_{\mathcal E}(G)\big)^2.
\]
Using the previously obtained relation
\[
\sum_{i=1}^{n}(\xi_i^{\mathcal L})^2
=\sum_{i=1}^{n}(\mathrm{Tr}_{\mathcal E}(v_i))^2+2e,
\]
we obtain
$\sum_{i=1}^{n}\eta_i^2
=\sum_{i=1}^{n}(\mathrm{Tr}_{\mathcal E}(v_i))^2
+2e-\frac{1}{n}\left(\sum_{i=1}^{n}\mathrm{Tr}_{\mathcal E}(v_i)\right)^2
=2E.$\qedhere
\end{enumerate}
\end{proof}
Now, we investigate spectral bounds for the eccentricity signless Laplacian matrix of a connected graph. In particular, we establish sharp inequalities for its largest eigenvalue in terms of the minimum, average, and maximum eccentric transmission parameters. 
 \begin{lemma}[Gershgorin Theorem]
 Let \( A = (a_{ij}) \) be a complex \( n \times n \) matrix, and let $\{\lambda_1, \lambda_2, \dots, \lambda_n\}$ denote its eigenvalues. Then
 $\lambda_p \in \smashoperator{\bigcup_{i=1}^{n}} \left\{ z : |z - a_{ii}| \leq \sum_{j \neq i} |a_{ij}| \right\},1\leq p \leq n.$
 \end{lemma}
 \begin{theorem}\label{bound theorerm}
 	Let $G$ be a connected graph. Then $
 	2\mathrm{Tr}_\mathcal{E}^{\min}\leq 2\overline{\mathrm{Tr}}_\mathcal{E} \leq \xi_1^\mathcal{Q} \leq 2\mathrm{Tr}_\mathcal{E}^{\max}.$
 \end{theorem}
 \begin{proof}
 	Using Rayleigh's quotient, we have	
 	$\xi_1^\mathcal{Q}(G)=\smashoperator{\max_{\mathbf{x}\neq 0}}\frac{\mathbf{x}^T\mathcal{E^{Q}}(G)\mathbf{x}}{\mathbf{x}^T\mathbf{x}}.$ If we take $\mathbf{x}=\mathbf{1}$, then we get $\xi_1^\mathcal{Q}(G)\geq2\overline{\mathrm{Tr}}_\mathcal{E}\geq2\mathrm{Tr}_\mathcal{E}^{\min}$. Also, by Gershgorin theorem the least upper bound for $\xi_1^\mathcal{Q}(G)$ is  $2\mathrm{Tr}_\mathcal{E}^{\max}$. Moreover, equality holds if and only if $G$ is $\mathcal{E}$-transmission regular graph i.e. $\mathbf{x}=\mathbf{1}$ is an eigenvector of $\mathcal{E}(G)$. 
 \end{proof}
 It is known that the eccentricity matrix of a connected graph may be either reducible or irreducible. By applying [Theorems~5.1 and~5.2 from~\cite{Minc H.}], we established that if $\mathcal{E}(G)$ is irreducible, then $\mathrm{Tr}_{\mathcal{E}}^{\max}(G) < \xi_1^\mathcal{Q} \leq 2\,\mathrm{Tr}_{\mathcal{E}}^{\max}(G)$, and if $\mathcal{E}(G)$ is reducible, then $\mathrm{Tr}_{\mathcal{E}}^{\max}(G) \leq  \xi_1^\mathcal{Q} \leq 2\,\mathrm{Tr}_{\mathcal{E}}^{\max}(G)$.
From the eccentric Laplacian spectrum and the eccentric signless Laplacian spectrum of a connected graph \( G \), one can determine the number of vertices of \( G \), since it equals the number of eigenvalues. The \(\mathcal{E}\)-Wiener index of \( G \) can be obtained as half of the sum of the eigenvalues of \(\mathcal{E^L}(G)\) (or \(\mathcal{E^Q}(G)\)). Using \ref{bound theorerm} and the fact that the sum of the eccentric signless Laplacian eigenvalues of $G$ is exactly the sum of the $\mathcal{E}$-transmission in $G$, one can conclude that $G$ is an $\mathcal{E}$-transmission regular graph if and only if the largest eigenvalue of $\mathcal{E^Q}(G)$ is equal to $\frac{2}{n}$ trace $(\mathcal{E^Q}(G))$.

Let $G$ be a connected graph. The characteristic polynomials of 
$\mathcal{E}(G)$, $\mathcal{E}^{\mathcal L}(G)$, and $\mathcal{E}^{\mathcal Q}(G)$ 
are denoted by $P_{\mathcal E}^{G}(t)$, $P_{\mathcal E^{\mathcal L}}^{G}(t)$, 
and $P_{\mathcal E^{\mathcal Q}}^{G}(t)$, respectively. Now we evaluate the characteristic polynomials $P_\mathcal{E}^G(t)$, $ P_\mathcal{E^L}^G(t)$ and $P_\mathcal{E^Q}^G(t)$ for some particular graphs.
		\begin{example}[The complete graph]	
	 Since $\mathcal{E}(K_n)=\mathcal{A}(K_n)$, the eccentric, the eccentric Laplacian, and the eccentric signless Laplacian spectra of the complete graph $K_n$ are, respectively, its adjacency, Laplacian, and signless Laplacian spectra, i.e.,
\[
\begin{aligned}
P_{\mathcal{E}}^{K_n}(t) &= (t-n+1)(t+1)^{n-1}; \\
P_{\mathcal{E}^{L}}^{K_n}(t) &= t(t-n)^{n-1}; \\
P_{\mathcal{E}^{Q}}^{K_n}(t) &= (t-2n+2)(t-n+2)^{n-1}.
\end{aligned}
\]
\end{example}
	\begin{example}[The complement of an edge]  
	 The complement of an edge is a graph obtained by removing an edge from a complete graph, i.e., $G=K_{n}-e$. Since $\mathcal{E}(K_n-e)=\mathcal{D}(K_n-e)$, by \cite{Aouchiche} we have the followings: 
    \[
\begin{aligned}
P_{\mathcal{E}}^{K_n-e}(t) 
&= \Bigl(t-\frac{n-1+\sqrt{(n-1)^2+8}}{2}\Bigr)
   \Bigl(t-\frac{n-1-\sqrt{(n-1)^2+8}}{2}\Bigr)
   (t+2)(t+1)^{n-3}; \\[4pt]
P_{\mathcal{E}^{L}}^{K_n-e}(t) 
&= t(t-n-2)(t-n)^{n-2}; \\[4pt]
P_{\mathcal{E}^{Q}}^{K_n-e}(t) 
&= \Bigl(t-\frac{3n-2+\sqrt{n^2-4n+20}}{2}\Bigr)
   \Bigl(t-\frac{3n-2-\sqrt{n^2-4n+20}}{2}\Bigr)
   (t-n+2)^{n-2}.
\end{aligned}
\]
	\end{example}
	\begin{example}[The complete bipartite graph]  
	The eccentric, the Laplacian eccentric, and the signless Laplacian characteristic polynomials of the complete bipartite graph $K_{m,n}$, $m,n\neq1$, are respectively,
\[
\begin{aligned}
P_{\mathcal{E}}^{K_{m,n}}(t) 
&= (t-2m+2)(t-2n+2)(t+2)^{m+n-2}; \\[4pt]
P_{\mathcal{E}^{L}}^{K_{m,n}}(t) 
&= t^2 (t-2m)^{m-1}(t-2n)^{n-1}; \\[4pt]
P_{\mathcal{E}^{Q}}^{K_{m,n}}(t) 
&= (t-4m+4)(t-4n+4)(t-2m+4)^{m-1}(t-2n+4)^{n-1}.
\end{aligned}
\]    
	\end{example} 
    \begin{example}[Cycle Graph]
Let $C_{2k}$ be a cycle of even order. Then
\[
\begin{aligned}
P_{\mathcal{E}}^{C_{2k}}(t) 
&= (t^2 - k^2)^k; \\
P_{\mathcal{E}^{L}}^{C_{2k}}(t) 
&= t^k (t - 2k)^k; \\
P_{\mathcal{E}^{Q}}^{C_{2k}}(t) 
&= t^k (t - 2k)^k.
\end{aligned}
\]

Now, let $C_{2k+1}$ be a cycle of odd order. Then
\[
\begin{aligned}
P_{\mathcal{E}}^{C_{2k+1}}(t) 
&= \prod_{i=1}^{2k+1}
\left(t - 2k \cos \frac{2\pi i}{2k+1}\right); \\
P_{\mathcal{E}^{L}}^{C_{2k+1}}(t) 
&= \prod_{i=1}^{2k+1}
\left(t - 2k + 2k \cos \frac{2\pi i}{2k+1}\right); \\
P_{\mathcal{E}^{Q}}^{C_{2k+1}}(t) 
&= \prod_{i=1}^{2k+1}
\left(t - 2k - 2k \cos \frac{2\pi i}{2k+1}\right).
\end{aligned}
\]
\end{example}
	\section{Equivalence with other Spectra}
 In this section, we establish the relationships among the spectra of various eccentricity matrices associated with graphs. Cvetkovi\'c and Simi\'c \cite{CvetkovicI, CvetkovicII, CvetkovicIII} made significant contributions to the study of spectral graph theory involving the signless Laplacian matrix. They discovered relationships among the adjacency, Laplacian, and signless Laplacian spectra for a regular graph and also demonstrated equivalence between the Laplacian and signless Laplacian spectra for bipartite graphs. Then M. Aouchiche and P. Hansen \cite{Aouchiche} extended the work for the distance Laplacian matrix. Our first finding is the equivalence between the spectrum of the eccentricity matrix $\mathcal{E}(G)$, the eccentricity Laplacian $\mathcal{E^L}(G)$, and the eccentricity signless Laplacian $\mathcal{E^Q}(G)$ on the set of $\mathcal{E}$-transmission regular graphs.
 
	 Let $G$ be an $\mathcal{E}$-transmission regular graph with $\mathcal{E}$-transmission degree $k$. Let $\{k=\xi_1\geq\xi_2\geq\dots\geq\xi_n\}$ be the eccentric spectrum of $G$; then $\{k-\xi_n\geq k-\xi_{n-1}\geq\dots\geq k-\xi_1=0\}=\{\xi_1^\mathcal{L}\geq\xi_2^\mathcal{L}\geq\dots\geq\xi_n^\mathcal{L}\}$ is the eccentric Laplacian spectrum of $G$, and $\{2k\geq k+\xi_2\geq\dots\geq k+\xi_n\}=\{\xi_1^\mathcal{Q}\geq\xi_2^\mathcal{Q}\geq\dots\geq\xi_n^\mathcal{Q}\}$ is the eccentric signless Laplacian spectrum of $G$. For example, see the following Petersen graph and its various spectra. The Petersen graph is $\mathcal{E}$-transmission regular graph with $\mathcal{E}$-transmission degree $12$.
\begin{center}
	\begin{minipage}{0.48\textwidth}
		\centering
		\begin{tikzpicture}[scale=1.5, every node/.style={circle,draw,fill=black,inner sep=1.5pt}]
			\node (1) at (90:1) {};
			\node (2) at (162:1) {};
			\node (3) at (234:1) {};
			\node (4) at (306:1) {};
			\node (5) at (18:1) {};
			
			\node (6) at (90:0.4) {};                        
			\node (7) at (162:0.4) {};
			\node (8) at (234:0.4) {};
			\node (9) at (306:0.4) {};
			\node (10) at (18:0.4) {};
			
			\draw (1) -- (2) -- (3) -- (4) -- (5) -- (1);
			
			\draw (6) -- (8) -- (10) -- (7) -- (9) -- (6);
			
			\draw (1) -- (6);
			\draw (2) -- (7);
			\draw (3) -- (8);
			\draw (4) -- (9);
			\draw (5) -- (10);
		\end{tikzpicture}
		
		\vspace{0.0000001cm}
		 {\small \textbf{Figure 1.} Petersen graph}
	\end{minipage}
	\hspace{0.1pt}
	\begin{minipage}{0.48\textwidth}
		\centering
		\begin{tabular}{|c|c|}
			\hline
			\textbf{$\mathcal{E}$-spectrum} & $12^{(1)} \quad 2^{(4)} \quad -4^{(5)}$ \\ \hline
			\textbf{$\mathcal{E^L}$-spectrum} & $16^{(5)} \quad 10^{(4)} \quad 0^{(1)}$ \\ \hline
			\textbf{$\mathcal{E^Q}$-spectrum} & $24^{(1)} \quad 14^{(4)} \quad 8^{(5)}$ \\ \hline
		\end{tabular}
		
		\vspace{0.3cm}
		 {\small \textbf{Table 1.} Different spectra of the Petersen graph}
	\end{minipage}
\end{center}

Recall that the eccentric \cite{Wang2019} (or $\mathcal{E}$-energy) of a graph is defined as $\sum_{i=1}^{n}|\xi_i|$,
where $\{\xi_1,\xi_2,\ldots,\xi_n\}$ are the eigenvalues of the eccentricity matrix. Similarly, we define the eccentric Laplacian energy (or $\mathcal{E}^{\mathcal L}$-energy) as
\[
\sum_{i=1}^{n}|\eta_i|,
\]	
where $\{\eta_1,\eta_2,\ldots,\eta_n\}$ are the auxiliary eigenvalues of the eccentric Laplacian matrix.

 In the next result, we show that for $\mathcal{E}$-transmission regular graphs, the eccentric energy and the eccentric Laplacian energy are equal.\begin{theorem}Let $G$ be a $\mathcal{E}$-transmission regular graph with $\mathcal{E}$-transmission degree $k$. Then the energy of $\mathcal{E}(G)$ and of $\mathcal{E^L}(G)$ are equal.
 \end{theorem}
	\begin{proof}
Since $\mathrm{Tr}_{\mathcal E}(v_i)=k$ for all $1\le i\le n$, we have
\[
k=\frac{1}{n}\sum_{j=1}^{n}\mathrm{Tr}_{\mathcal E}(v_j).
\]
The auxiliary eigenvalues of $\mathcal{E}^{\mathcal L}(G)$ are given by
\[
\eta_i=\xi_i^{\mathcal L}-\frac{1}{n}\sum_{j=1}^{n}\mathrm{Tr}_{\mathcal E}(v_j),
\qquad 1\le i\le n.
\]
Substituting the value of the average trace, we obtain
\[
\eta_i=\xi_i^{\mathcal L}-k.
\]
Since $\xi_i^{\mathcal L}=k-\xi_{n+1-i}$ for $1\le i\le n$, it follows that
$
\eta_i=\xi_{n+1-i}.
$
Hence,
\[
\sum_{i=1}^{n}|\eta_i|
=\sum_{i=1}^{n}|\xi_{n+1-i}|
 =\sum_{i=1}^{n}|\xi_i|. \qedhere
\]
\end{proof}
	 
	Next, we find equivalence between Laplacian spectra and eccentric Laplacian spectra (signless Laplacian spectra and eccentric signless Laplacian spectra) for $G\vee K_1$, where $G$ is a regular (adjacency) graph of diameter $2$. But first, we recall the following results.
\begin{lemma}\cite{Wang2018}
	 Let $G$ be an $r$-regular (adjacency) graph with diameter $2$, and let $\{r=\lambda_1,\lambda_2,\dots,\lambda_n\}$ be the adjacency spectrum of $G$. Then $\mathcal{E}(G\vee K_1)=\begin{bmatrix}2\mathcal{A}(\bar{G}) &  \mathbf{1}\\\mathbf{1}^T & 0\end{bmatrix}$ and $\{(n-r-1)\pm\sqrt{(n-r-1)^2+1},-2(\lambda_2+1),\dots,-2(\lambda_n+1)\}$ is the eccentric spectrum of $G\vee K_1$.
\end{lemma}
\noindent 

	\begin{lemma}\label{Coronal}\cite{Wang2018}
		  Let $A$ be an $n\times n$ order matrix, and let each column sum of $A$ be constant (say $\alpha$). Then $\mathbf{1}^T(tI-A)^{-1}\mathbf{1}=\frac{n}{t-\alpha}$.
	\end{lemma}
	\begin{theorem}
 Let $G$ be an $r$-regular (adjacency) graph of diameter $2$. Let $\sigma(\mathcal{L}(G))=\{\mu_1,\mu_2,\dots,\mu_n\}, \sigma(\mathcal{Q}(G))=\{q_1,q_2,\dots,q_n\},$ and let $\{\mathbf{u}_1,\mathbf{u}_2,\dots,\mathbf{u}_n=\mathbf{1}\}$ and $\{\mathbf{1}=\mathbf{q}_1,\mathbf{q}_2,\dots,\mathbf{q}_n\}$ are the corresponding eigenvectors of $\mathcal{L}(G)$ and $\mathcal{Q}(G)$, respectively. Then the following holds.
\begin{enumerate}
\item[(i)]
$
\sigma(\mathcal{E}^{\mathcal L}(G\vee K_1))
=
\{2n+1-2\mu_1,\,2n+1-2\mu_2,\dots,2n+1-2\mu_{n-1},\,0,\,n+1\}.$
Moreover,
\[
\begin{bmatrix}
\mathbf{u}_1\\
0
\end{bmatrix},
\begin{bmatrix}
\mathbf{u}_2\\
0
\end{bmatrix},
\dots,
\begin{bmatrix}
\mathbf{u}_{n-1}\\
0
\end{bmatrix},
\begin{bmatrix}
\mathbf{u}_n\\
1
\end{bmatrix},
\begin{bmatrix}
\mathbf{u}_n\\
-n
\end{bmatrix}
\]
are the corresponding eigenvectors of $\mathcal{E}^{\mathcal L}(G\vee K_1)$.

\item[(ii)]
$\sigma(\mathcal{E}^{\mathcal Q}(G\vee K_1))
=
\{2n-3-q_2,\,2n-3-q_3,\dots,2n-3-q_n,\,t_1,\,t_2\},$ where $t_1$ and $t_2$ are the zeros of the quadratic polynomial $t^2-(5n-4r-3)t+4n(n-r-1).$
Furthermore,
\[
\begin{bmatrix}
\mathbf{q}_2\\
0
\end{bmatrix},\begin{bmatrix}
\mathbf{q}_3\\
0
\end{bmatrix}
\dots,
\begin{bmatrix}
\mathbf{q}_n\\
0
\end{bmatrix},
\begin{bmatrix}
\mathbf{q}_1\\
\alpha_1
\end{bmatrix},
\begin{bmatrix}
\mathbf{q}_1\\
\alpha_2
\end{bmatrix}
\]
are the corresponding eigenvectors, where 
$
\alpha_1=\frac{n}{t_1-n}, \ \alpha_2=-\frac{n}{\alpha_1}.$
\end{enumerate}
\end{theorem}

\begin{proof}
    
\begin{enumerate}
    \item[(i)] Since
\[
\mathcal{E}(G\vee K_1)=
\begin{bmatrix}
2J-2I-2\mathcal{A}(G) & \mathbf{1}\\
\mathbf{1}^{T} & 0
\end{bmatrix},
\]

\[
\operatorname{Diag}\!\big(\mathrm{Tr}_{\mathcal E}(G\vee K_1)\big)=
\begin{bmatrix}
(2n-2r-1)I & \mathbf{0}\\
\mathbf{0}^{T} & n
\end{bmatrix},
\]
it follows that
\[
\mathcal{E}^{\mathcal L}(G\vee K_1)=
\begin{bmatrix}
(2n+1)I-2J-2\mathcal{L}(G) & -\mathbf{1}\\
-\mathbf{1}^{T} & n
\end{bmatrix}.
\]

Hence,
\[
\det\!\big(tI-\mathcal{E}^{\mathcal L}(G\vee K_1)\big)
=\det
\begin{bmatrix}
tI-\big((2n+1)I-2J-2\mathcal{L}(G)\big) & \mathbf{1}\\
\mathbf{1}^{T} & t-n
\end{bmatrix}.
\]

Applying the Schur complement and Lemma~\ref{Coronal}, we obtain the following:
\[
\begin{aligned}
\det\!\big(tI-\mathcal{E}^{\mathcal L}(G\vee K_1)\big)
&=\det\!\big[tI-\big((2n+1)I-2J-2\mathcal{L}(G)\big)\big] \\
&\quad \times \Big[(t-n)-\mathbf{1}^{T}
\big(tI-\big((2n+1)I-2J-2\mathcal{L}(G)\big)\big)^{-1}
\mathbf{1}\Big].
\end{aligned}\]
Consequently,
\[
\det\!\big(tI-\mathcal{E}^{\mathcal L}(G\vee K_1)\big)
=\prod_{i=1}^{n-1}\big[t-(2n+1-2\mu_i)\big]\; t\,(t-n-1).
\]
Therefore,
$\{\,2n+1-2\mu_1,\;2n+1-2\mu_2,\;\ldots,\;2n+1-2\mu_{n-1},\;0,\;n+1\,\}$
the eccentric Laplacian spectrum of $G\vee K_1$.

One may see that for $1\leq i\leq n-1$ \[
\mathcal{E}^{\mathcal L}(G\vee K_1)\begin{bmatrix}
\mathbf{u}_i\\
0
\end{bmatrix}=
\begin{bmatrix}
(2n+1)I-2J-2\mathcal{L}(G) & -\mathbf{1}\\
-\mathbf{1}^{T} & n
\end{bmatrix}\begin{bmatrix}
\mathbf{u}_i\\
0
\end{bmatrix}=(2n+1-2\mu_i)\begin{bmatrix}
\mathbf{u}_i\\
0
\end{bmatrix}\]
 and corresponding to eigenvalue $0$ \[
\mathcal{E}^{\mathcal L}(G\vee K_1)\begin{bmatrix}
\mathbf{u}_n\\
1
\end{bmatrix}=
\begin{bmatrix}
(2n+1)I-2J-2\mathcal{L}(G) & -\mathbf{1}\\
-\mathbf{1}^{T} & n
\end{bmatrix}\begin{bmatrix}
\mathbf{u}_n\\
1
\end{bmatrix}=0\begin{bmatrix}
\mathbf{u}_n\\
1
\end{bmatrix}.\]
  Let $\begin{bmatrix}\mathbf{x}\\\alpha\end{bmatrix}$be an eigenvector corresponding to the eigenvalue $n+1$, where $\mathbf{x}\in \mathbb{R}^{n}$ and $\alpha\in \mathbb{R}$. Since $\begin{bmatrix}\mathbf{u}_i\\\alpha\end{bmatrix} \perp  \begin{bmatrix}\mathbf{x}\\\alpha\end{bmatrix}\perp \begin{bmatrix}\mathbf{u}_n\\1\end{bmatrix}$ for $1\leq i\leq n-1$, implies that $\mathbf{x}=\mathbf{u}_n$ and $\alpha=-n$.
\item[(ii)] Since,
\[
\mathcal{E}^{\mathcal Q}(G\vee K_1)=
\begin{bmatrix}
(2n-3)I+2J-2\mathcal{Q}(G) & \mathbf{1}\\
\mathbf{1}^T & n
\end{bmatrix}.
\]
Hence,
\[
\det\!\big(tI-\mathcal{E}^{\mathcal Q}(G\vee K_1)\big)
=
\det
\begin{bmatrix}
tI-\big((2n-3)I+2J-2\mathcal{Q}(G)\big) & -\mathbf{1}\\
-\mathbf{1}^T & t-n
\end{bmatrix}.
\]
Using the Schur complement together with Lemma~\ref{Coronal}, we obtain
\[
\det\!\big(tI-\mathcal{E}^{\mathcal Q}(G\vee K_1)\big)
=
\prod_{i=2}^{n}\!\left[t-(2n-3-q_i)\right]
\left(t^2-(5n-4r-3)t+(4n^2-4rn-4n)\right).
\]
Therefore, the eccentric signless Laplacian spectrum of $G\vee K_1$ is
\[
\{2n-3-q_2,\,2n-3-q_3,\,\dots,\,2n-3-q_n,\,t_1,\,t_2\},
\]
where $t_1$, $t_2$ are the zeros of the quadratic polynomial
$t^2-(5n-4r-3)t+4n(n-r-1).$  In a similar manner, it can be verified that the vectors
\[
\begin{bmatrix}
\mathbf{q}_2\\
0
\end{bmatrix},
\begin{bmatrix}
\mathbf{q}_3\\
0
\end{bmatrix},
\dots,
\begin{bmatrix}
\mathbf{q}_n\\
0
\end{bmatrix},
\begin{bmatrix}
\mathbf{q}_1\\
\alpha_1
\end{bmatrix},
\begin{bmatrix}
\mathbf{q}_1\\
\alpha_2
\end{bmatrix}
\]
are corresponding eigenvectors of $\mathcal{E}^{\mathcal Q}(G\vee K_1),$ where
$\alpha_1=\frac{n}{t_1-n}
\quad \text{and} \quad
\alpha_2=-\frac{n}{\alpha_1}. 
$\qedhere
\end{enumerate}
\end{proof}
\begin{remark}
  If $r=\frac{n-1}{2}$ then $G\vee K_1$ is an $\mathcal{E}$-transmission regular graph of $\mathcal{E}$-transmission degree $n$. Then $\begin{bmatrix}
	\mathbf{q}_1\\
	1
\end{bmatrix}$,$\begin{bmatrix}
\mathbf{q}_1\\
-n
\end{bmatrix}$ are the eigenvectors of $\mathcal{E^{Q}}(G\vee K_1)$ corresponding to eigenvalues $2n,n-1$ $(\text{say} ~ t_1,t_1)$ respectively.  
\end{remark}
\begin{center}
	\begin{minipage}{0.45\textwidth}
		\centering
		\begin{tikzpicture}[scale=1.5, every node/.style={circle, draw, fill=black, inner sep=1pt}]
			\node (1) at (0,0) [label=below:$1$] {};
			\node (2) at (0,1) [label=above:$2$] {};
			\node (3) at (-0.95,0.3) [label=left:$3$] {};
			\node (4) at (-0.6,-0.8) [label=below:$4$] {};
			\node (5) at (0.6,-0.8) [label=below:$5$] {};
			\node (6) at (0.95,0.3) [label=right:$6$] {};
			
			\draw (1) -- (2);
			\draw (1) -- (3);
			\draw (1) -- (4);
			\draw (1) -- (5);
			\draw (1) -- (6);
			\draw (2) -- (3);
			\draw (3) -- (4);
			\draw (4) -- (5);
			\draw (5) -- (6);
			\draw (6) -- (2);
		\end{tikzpicture}
		
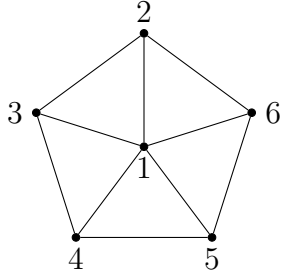
		{\small \textbf{Figure 2.} $C_5\vee K_1$}
	\end{minipage}
	\hfill
	\begin{minipage}{0.45\textwidth}
		\centering
		\begin{tikzpicture}[scale=1.5, every node/.style={circle, draw, fill=black, inner sep=1pt}]
			\node (1) at (0,0) [label=below:$1$] {}; 
			\node (2) at (-1,1) [label=above:$2$] {};
			\node (3) at (1,1) [label=above:$3$] {};
			\node (4) at (1,-1) [label=below:$4$] {};
			\node (5) at (-1,-1) [label=below:$5$] {};
			
			\draw (2) -- (3) -- (4) -- (5) -- (2);
			
			\draw (1) -- (2);
			\draw (1) -- (3);
			\draw (1) -- (4);
			\draw (1) -- (5);
		\end{tikzpicture}
		
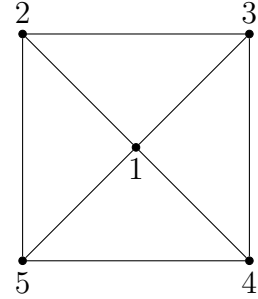
		{\small \textbf{Figure 3.} $C_4\vee K_1$}
	\end{minipage}
\end{center}
The first graph is a $\mathcal{E}$-transmission regular graph but not a degree regular graph with the smallest order, and the second graph is neither degree regular nor a $\mathcal{E}$-transmission regular graph.

In \cite{Aouchiche}, the authors established a relationship between the Laplacian spectrum and the distance Laplacian spectrum of graphs with diameter $2$. Motivated by this line of research, we establish a similar correspondence between the Laplacian (signless Laplacian) spectrum and the eccentric Laplacian (eccentric signless Laplacian) spectrum for this class of graphs. First, we find the structure of the eccentric Laplacian matrix of a connected graph of diameter $2$.

   \begin{theorem}\label{eccentric laplacian of diameter 2}
   	Let $G$ be a connected graph of order $n$ with diameter $2$. Then $$ \mathcal{E^L}(G)=\begin{bmatrix}
   	 	nI-J  & -J \\
   	 	-J  & (2n-u)I-2J+2\mathcal{L}(G^\prime)
   	 \end{bmatrix},$$ where $G^\prime$ is induced subgraph of $V(G)\setminus U(G)$, and $u=|U(G)|$.
\end{theorem}
   \begin{proof}
Since $\operatorname{diam}(G)=2$, we have $e(v_i)\in\{1,2\}$ for all $v_i\in V(G)$. Consequently,
\[
\mathrm{Tr}_{\mathcal{E}}(i)=
\begin{cases}
n-1, & \text{if } e(v_i)=1,\\[4pt]
2\big((n-1)-\deg(i)\big)+|U(G)|, & \text{if } e(v_i)=2.
\end{cases}
\]

Let $V(G)=\{v_1,\dots,v_u,v_{u+1},\dots,v_n\}$ be such that 
$\deg(v_i)=n-1$ for $1\leq i\leq u$ and $\deg(v_i)<n-1$ for $u+1\leq i\leq n$, 
where $u=|U(G)|$.
Then
\[
\operatorname{Diag}\!\left(\mathrm{Tr}_{\mathcal{E}}(G)\right)=
\begin{bmatrix}
(n-1)I  & O \\
O  & \left[2\{(n-1)-\deg(v_i)+u\}\right]I 
\end{bmatrix}
\]
and
\[
\mathcal{E}(G)=
\begin{bmatrix}
J-I & J \\
J   & 2J-2I-\mathcal{A}(G')
\end{bmatrix},
\]
where $G'$ is the induced subgraph of $V(G)\setminus U(G)$.
Thus,
\[
\mathcal{E^L}(G)=
\begin{bmatrix}
nI-J & -J \\
-J &
(2n+u)I-2J+2\mathcal{A}(G')-2\operatorname{Diag}\!\big(\deg_G(V(G)\setminus U(G))\big)
\end{bmatrix}.
\]

Note that
\[
\operatorname{Diag}\!\big(\deg_G(V(G)\setminus U(G))\big)
= uI+\operatorname{Diag}\!\big(\deg_{G'}(V(G)\setminus U(G))\big).
\]

Hence,
\[
\mathcal{E^L}(G)=
\begin{bmatrix}
nI-J & -J\\
-J & (2n-u)I-2J-2\mathcal{L}(G')
\end{bmatrix}. \qedhere
\]
\end{proof}

  \begin{corollary}
Let $G$ be a connected graph of diameter $2$, without any universal vertex, on $n$ vertices, and let $\bar{G}$ denote its complement; then $\mathcal{L}(\bar{G})=\frac{1}{2}\mathcal{E^L}(G).$ Proposition~2.1 of \cite{Wang2022} also follows from Theorem~\ref{eccentric laplacian of diameter 2}.
  \end{corollary}
 \begin{theorem}\label{main_theorem} 	
 Let $G$ be a connected graph of diameter $2$, without any universal vertex, of order $n$. Let $\sigma(\mathcal{L}(G))=\{\mu_1, \mu_2 ,\dots, \mu_{n-1}, \mu_n=0\}$, and $\sigma(\mathcal{Q}(G)=\{2r=q_1, q_2,\dots, q_n\}.$ Then
    \begin{enumerate}
        \item[(i)] $\sigma(\mathcal{E^L}(G))=\{2n-2\mu_{n-1}, 2n-2\mu_{n-2}, \dots, 2n-2\mu_1, \xi_n^\mathcal{L}=0\}$, and for each eigenvalue, the corresponding eigenvectors are also the same.
        \item[(ii)] Furthermore, if $G$ is an $r$-regular (adjacency) graph, then $\sigma(\mathcal{E^Q}(G)=\{4(n-r-1),2n-4-2q_2,\dots,2n-4-2q_n\}$, and for each eigenvalue, the corresponding eigenvectors are also the same.
      \end{enumerate}
 \end{theorem}
 \begin{proof}
\begin{enumerate}
    \item[(i)] For a connected graph of diameter $2$ without universal vertices, the $\mathcal{E}$-transmission of each vertex $v_i\in V(G)$ is $\mathrm{Tr}_\mathcal{E}(v_i)=2(n-1-deg(v_i))$, where $deg(v_i)$ denotes the degree of $v_i$. Also, the eccentricity matrix for such graphs is $\mathcal{E}(G)=2J-2I-2\mathcal{A}(G)$. Thus, the eccentric Laplacian matrix can be written as 
    \begin{align*}
\mathcal{E}^{\mathcal L}(G)
&= (2n-2)I-2\operatorname{Diag}(\operatorname{Deg}(G))-2J+2I+2\mathcal{A}(G) \\
&= 2nI-2J-2\mathcal{L}(G).
\end{align*}
 Clearly, $\mathcal{E^L}(G)\mathbf{1}=0\mathbf{1}$. Consider any nonzero Laplacian eigenvalue $\mu_i$ with $1\leq i \leq n-1$ and let $\mathbf{x}_i$ denote a Laplacian eigenvector for $\mu_i$. Since $\mathcal{L}(G)$ is a symmetric matrix and $\mathbf{1}$ is an eigenvector for $\mu_n=0$, all other eigenvectors are orthogonal to $\mathbf{1}$. Thus,
$\mathcal{E^L}(G)\mathbf{x}_i=2n\mathbf{x}_i-2\mu_i\mathbf{x}_i=(2n-2\mu_i)\mathbf{x}_i$, for $1\leq i \leq n-1$.
\item[(ii)] Note that, \begin{align*}
\mathcal{E}^{\mathcal Q}(G)
&= (2n-2)I-2\operatorname{Diag}(\operatorname{Deg}(G))+2J-2I-2\mathcal{A}(G) \\
&= (2n-4)I+2J-2\mathcal{Q}(G).
\end{align*} Since $\mathcal{Q}(G)\mathbf{1}=2r\mathbf{1}$ implies that $\mathcal{E^Q}(G)\mathbf{1}=4(n-r-1)\mathbf{1}$. Now, for any other signless Laplacian eigenvalue $q_i$ where $2\leq i \leq n$, let $\mathbf{y}_i$ denote the corresponding eigenvector. Then $\mathcal{E^Q}(G)\mathbf{y}_i=(2n-4-2q_i)\mathbf{y}_i$, $2\leq i \leq n$. \qedhere
\end{enumerate} 
 \end{proof}
Graham and Lov\'asz computed the inverse of the distance matrix of a tree \cite{Graham1978} in terms of its Laplacian matrix. Later, Bapat et al. \cite{Bapat2005} extended this result to weighted trees. A relationship between the eccentric Laplacian matrix and the Laplacian matrix of a graph with diameter $2$ can be obtained using Theorem~\ref{main_theorem}.
 \begin{corollary}
 		Let $G$ be a connected graph of diameter $2$, without any universal vertex, of order $n$. Then 
 		$\mathcal{E^L}(G)=2nI-RS^T\mathcal{L}(G)SR^T,$ where $R$ and $S$ are orthogonal matrices defined by $R=[\mathbf{r}_1,\mathbf{r}_2,\dots, \mathbf{r}_{n-1},\mathbf{r}_n]$ and $S=[\mathbf{r}_{n-1},\mathbf{r}_{n-2},\dots,\mathbf{r}_1,\mathbf{r}_n]$ with $\mathbf{r}_1,\mathbf{r}_2,\dots, \mathbf{r}_{n-1},\mathbf{r}_n$ being the eigenvectors of $\mathcal{E^L}(G)$ corresponding to $\{\xi_1^\mathcal{L},\xi_2^\mathcal{L},\dots,\xi_n^\mathcal{L}\}$, respectively.
 \end{corollary}
\begin{proof}
Since $\mathbf{r}_1,\mathbf{r}_2,\dots,\mathbf{r}_{n}$ are eigenvectors of 
$\mathcal{E}^{\mathcal L}(G)$ and $\mathcal{L}(G)$ corresponding to the eigenvalues 
$\{\xi_1^{\mathcal L},\xi_2^{\mathcal L},\dots,\xi_n^{\mathcal L}\}$ and 
$\{\mu_{n-1},\mu_{n-2},\dots,\mu_1, \mu_n\}$, respectively, we have
\begin{align*}
\mathcal{E}^{\mathcal L}(G)
&= R\,\operatorname{Diag}\!\left(\xi_1^{\mathcal L},\xi_2^{\mathcal L},\dots,\xi_n^{\mathcal L}\right)R^{T} \\
&= R\!\left(2nI-\operatorname{Diag}(\mu_{n-1},\mu_{n-2},\dots,\mu_1,\mu_n)\right)R^{T} \\
&= 2nI - R S^{T}\mathcal{L}(G) S R^{T}.\qedhere
\end{align*}
\end{proof}
 \section{Spectral Characterization of $\mathcal{E}$-Bipartite Graphs}
In analogy with bipartite graphs in spectral graph theory and their characterization via matrix decompositions and the spectrum of the associated matrices, we define a new class of graphs based on the structure of the eccentricity matrix.
\begin{definition}
A connected graph $G$ is said to be $\mathcal{E}$-bipartite if, under a suitable labeling of the vertices, its eccentricity matrix can be written in the form of $\begin{bmatrix}
    O & B\\
    B^T & O
    
\end{bmatrix}.$
\end{definition}
\begin{example}
    Tree with odd diameter, cycle of even length, n-barbell graph, cocktail-party graph on $2n$ vertices, and lollipop graphs are such standard $\mathcal{E}$-bipartite graphs.
\end{example}
\begin{lemma}\label{bipartite lemma}
    The spectrum of the matrix in the form $A=\begin{bmatrix}
        O & X\\
        X^T & O
    \end{bmatrix}$, is symmetric about zero.
\end{lemma}
\begin{proof}
     To equate the dimensions of $O$ and $X$, we adjoin zero rows and zero columns as necessary. Without loss of generality, we henceforth suppose that $O$ is of strictly larger order than $X$. This modification doesn't impact the property we're proving, as appending zero rows and columns to A simply extends it without altering its core structure or the relevant mathematical characteristics. Then If $\begin{bmatrix}
        x^{(1)}\\
        x^{(2)}
    \end{bmatrix}$ is eigenvector correponding $\alpha$ of $A$ then $\begin{bmatrix}
        x^{(1)}\\
        -x^{(2)}
    \end{bmatrix}$ is eigenvector correponding $-\alpha$.
\end{proof}
\begin{lemma}\label{similar matrices}
A matrix $\begin{bmatrix}
    A & B\\
    C & D
\end{bmatrix}$ is similar to $\begin{bmatrix}
    A & -B\\
    -C & D
\end{bmatrix}$.
\end{lemma}
\begin{proof}
  $\begin{bmatrix}
       I & O\\
       O & -I
   \end{bmatrix}\begin{bmatrix}
    A & B\\
    C & D
\end{bmatrix} \begin{bmatrix}
       I & O\\
       O & -I
   \end{bmatrix}^T=\begin{bmatrix}
    A & -B\\
    -C & D
\end{bmatrix}.$
\end{proof}
Spectral properties of matrices associated with graphs reveal deep structural information about the graph. For instance, it is well-known that the Laplacian spectrum and the signless Laplacian spectrum of a graph coincide if and only if the graph is bipartite \cite{Cvetkovic2007}. Inspired by such characterizations, we present a similar spectral criterion in the context of $\mathcal{E^L}(G)$ and $\mathcal{E^Q}(G)$, offering new insight into the interplay between graph structure and spectrum. A matrix $A$ is said to be reducible if it is permutationally similar to a block upper triangular matrix; otherwise, it is irreducible.
\begin{theorem}
Let $G$ be a connected graph on $n$ vertices such that $\mathcal{E}(G)$ is an irreducible matrix. Then the following statements are equivalent:
    \begin{enumerate}
        \item[(i)] $G$ is $\mathcal{E}$-bipartite;
        \item[(ii)] $\mathcal{E}$-spectrum is symmetric about origin;
        \item[(iii)] $-\rho(\mathcal{E}(G))\in\sigma(\mathcal{E}(G))$;
        \item[(iv)]
        $\mathcal{E^L}(G)$ and $\mathcal{E^Q}(G)$ are similar matrices. 
        
    \end{enumerate}
\end{theorem}

\begin{proof}
 From Lemma \ref{bipartite lemma} (i)$\implies$(ii). Since for any nonnegative matrix, the spectral radius is an eigenvalue, so (ii)$\implies$(iii). 
 
 (iii)$\implies$ (i): Let $\mathcal{E}(G)=\begin{bmatrix}
     B_1 & B_2\\
     B_{2}^T & B_3 
 \end{bmatrix}$. As $\mathcal{E}(G)$ is a non-negative irreducible matrix and  $- \rho(\mathcal{E}(G))$ is an eigenvalue of $\mathcal{E}(G)$, there exist unit vectors $\mathbf{u}>0 $ and $\mathbf{v}$ such that $$\mathcal{E}(G)\mathbf{u}=\rho(\mathcal{E}(G))\mathbf{u}~~\text{and}~~\mathcal{E}(G)\mathbf{v}=-\rho(\mathcal{E}(G))\mathbf{v}.$$ This gives $|\mathcal{E}(G)\mathbf{v}|=\rho(\mathcal{E}(G))|\mathbf{v}|$ which implies $\rho(\mathcal{E}(G))|\mathbf{v}|\leq \mathcal{E}(G))|\mathbf{v}|$. Since $|\mathbf{v}|$ is a non-negative unit vector, we have  $|\mathbf{v}|^T\rho(\mathcal{E}(G))|\mathbf{v}|\leq |\mathbf{v}|^T\mathcal{E}(G))|\mathbf{v}|$. So $\rho(\mathcal{E}(G))\leq |\mathbf{v}|^T\mathcal{E}(G))|\mathbf{v}|$. From the fact $\rho(\mathcal{E}(G))=\max\{\mathbf{x}^T\mathcal{E}(G)\mathbf{x}:\mathbf{x} ~\text{is a unit vector}\}$ and the maximal value is achieved in the eigenvector of $\mathcal{E}(G)$ associated with $\rho(\mathcal{E}(G))$. So we get $\mathcal{E}(G)|\mathbf{v}|=\rho(\mathcal{E}(G))|\mathbf{v}|$. As $u$ is the unique unit eigenvector that corresponds to $\rho(\mathcal{E}(G))$, we must have $|\mathbf{v}|=\mathbf{u}$.  This implies that there exists a signature matrix $S$ such that $\mathbf{v}=S\mathbf{u}$. Without loss of generality, we now assume that $S$ is of the form  $S=\begin{bmatrix}
     I_k & O\\
     O & -I_{n-k}
 \end{bmatrix}$. If we take $\mathbf{u}=\begin{bmatrix}
     u_1 \\ u_2
 \end{bmatrix}$, then  $\mathbf{v}=\begin{bmatrix}
     u_1\\-u_2
 \end{bmatrix}$. By solving the following two equations $\begin{bmatrix}
     B_1 & B_2\\
     B_{2}^T & B_3 
 \end{bmatrix}\begin{bmatrix}
     u_1 \\ u_2
 \end{bmatrix}=\rho(\mathcal{E}(G))\begin{bmatrix}
     u_1 \\ u_2
 \end{bmatrix}$ and  $\begin{bmatrix}
     B_1 & B_2\\
     B_{2}^T & B_3 
 \end{bmatrix}\begin{bmatrix}
     u_1 \\ -u_2
 \end{bmatrix}=-\rho(\mathcal{E}(G))\begin{bmatrix}
     u_1 \\ -u_2
 \end{bmatrix}$, we get $B_1$ and $B_3$ are zero matrices. Hence $\mathcal{E}(G)=\begin{bmatrix}
     O & B_2\\
     B_{2}^T & O
 \end{bmatrix}$. 
 
 (i) $\Longleftrightarrow$ (iv): (i)$\implies$(iv) holds by Lemma \ref{similar matrices}. Now we will prove (iv) $\implies $ (i). Assume that $\mathcal{E^L}(G)$ and  $\mathcal{E^Q}(G)$ are similar matrices. As $\mathcal{E}(G)$  is irreducible, $\mathcal{E^Q}(G)$ is an irreducible non-negative matrix. By Perron-Frobenius theorem, there exists a unique unit positive vector $\mathbf{w}$ such that $\mathcal{E^Q}(G)\mathbf{w}=\rho(\mathcal{E^Q}(G))\mathbf{w}$.  Since $\mathcal{E^L}(G)$ and  $\mathcal{E^Q}(G)$ are similar matrices,  there is a unit vector $\mathbf{z}$ satisfies $\mathcal{E^L}(G)\mathbf{z}=\rho(\mathcal{E^Q}(G))\mathbf{z}$. This gives \begin{equation*}
    \rho(\mathcal{E^Q}(G))|\mathbf{z}|= |\mathcal{E^L}(G)\mathbf{z}|\leq|\mathcal{E^L}(G)||\mathbf{z}| \leq|\mathcal{E^L}(G)||\mathbf{z}|\leq|\mathcal{E^Q}(G)||\mathbf{z}|.
 \end{equation*}
 As $|z|\geq 0$ and $||z||=1$, we get $\rho(\mathcal{E^Q}(G))=|\mathbf{z}|^T\rho(\mathcal{E^Q}(G))|\mathbf{z}|\leq|\mathbf{z}|^T\mathcal{E^Q}(G)|\mathbf{z}|\leq\rho(\mathcal{E^Q}(G))$. Hence $\mathcal{E^Q}(G)|\mathbf{z}|=\rho(\mathcal{E^Q}(G)|\mathbf{z}|$ which implies $|\mathbf{z}|=\mathbf{w}$. So, $\mathbf{z}=S\mathbf{w}$ for some signature matrix $S$. No loss of generality, we assume that $S=\begin{bmatrix}
     I_k & O\\
     O & -I_{n-k}
 \end{bmatrix}$. If we set $\mathbf{w}=\begin{bmatrix}
     w_1 \\ w_2
 \end{bmatrix}$, then  $\mathbf{z}=S\mathbf{w}= \begin{bmatrix}
     w_1\\-w_2
 \end{bmatrix}$. Let eccentricity matrix $\mathcal{E}(G)=\begin{bmatrix}
    B_1& B_2\\
     B_{2}^T & B_3   
 \end{bmatrix}$ and  let transmission matrix $D=\begin{bmatrix}
    D_1& 0\\
     0 & D_2  
 \end{bmatrix}.$
As $\mathcal{E^Q}(G))\mathbf{w}= \rho(\mathcal{E^Q} (G)) \mathbf{w}$ and $\mathcal{E^L}(G))\mathbf{z}= -\rho(\mathcal{E^Q} (G)) \mathbf{z}$, we have 
\begin{tabular}{l l}
   $B_1 w_1 +D_1 w_1 +B_2 w_2 =\rho(\mathcal{E^Q} (G)) w_1$ &  and ~$B^T_2 w_1 +B_3 w_2 +D_2 w_2=\rho(\mathcal{E^Q} (G)) w_2$\\
   $B_1 w_1 -D_1 w_1 -B_2 w_2 =-\rho(\mathcal{E^Q} (G)) w_1$ &  and ~$B^T_2 w_1 -B_3 w_2 +D_2 w_2=\rho(\mathcal{E^Q} (G)) w_2$.
\end{tabular}

By solving the above equations, we get
 $B_1$ and $B_3$ are zero matrices, hence $G$ is a $\mathcal{E}$-bipartite graph. This completes the proof.
\end{proof}
 
\section*{Conclusion}
This work establishes the eccentricity Laplacian and signless Laplacian as powerful spectral tools, revealing equivalences with the eccentricity spectrum for $\mathcal{E}$-transmission regular graphs, graphs with diameter $2$, and a full characterization of $\mathcal{E}$-bipartite graphs through spectral symmetry and matrix similarity.

 
\end{document}